\documentclass[11pt, a4paper]{article}

\usepackage[english]{babel}
\usepackage{enumitem}
\usepackage{amsmath}
\usepackage{amsthm}
\usepackage{amssymb}
\usepackage{tikz}
\usepackage{subcaption}
\usetikzlibrary{positioning,intersections,through,decorations.pathreplacing,calligraphy}
\usepackage{lastpage}
\usepackage{thmtools}
\usepackage{thm-restate}
\usepackage{tablefootnote}
\usepackage{footnote}
\usepackage[nocompress]{cite}

\newcommand{\YZ}[1]{{#1}}

\newcommand{\GG}[1]{{#1}}
\newcommand{\MA}[1]{{#1}}

\newcommand{\AY}[1]{{#1}}

\usetikzlibrary{decorations.pathmorphing}
\tikzstyle{vertexB}=[circle,draw, minimum size=18pt, scale=0.55, inner sep=0.0pt]
\tikzstyle{vertexS}=[circle,draw, minimum size=18pt, scale=0.52, inner sep=0.0pt]
\tikzstyle{arc}=[->,line width=0.03cm]

\theoremstyle{definition}

\newcommand{\2}{\vspace{0.2cm}}

\newtheorem{thm}{Theorem}[section]
\newtheorem{lemma}{Lemma}[section]

\newtheorem{claim}{Claim}

\newtheorem{prop}{Proposition}[section]
\newtheorem{conj}{Conjecture}

\usepackage[margin=3cm]{geometry}

\author{ Gregory Gutin\thanks{Department of Computer Science. Royal Holloway University of London, UK, and School of Mathematical Sciences and LPMC, Nankai University, Tianjin 300071, China.  {\tt g.gutin@rhul.ac.uk}}    \hspace{2mm} Mads Anker Nielsen\thanks{Department of Mathematics and
Computer Science, University of Cologne, Germany. {\tt m.nielsen@uni-koeln.de}} \hspace{2mm}  Anders Yeo\thanks {Department of Mathematics and Computer Science, University of Southern Denmark, Denmark, and Department of Mathematics and Applied Mathematics, University of Johannesburg, South Africa. {\tt yeo@imada.sdu.dk}} \hspace{2mm}  Yacong Zhou\thanks{Department of Computer Science. Royal Holloway University of London, UK. {\tt Yacong.Zhou.2021@live.rhul.ac.uk}} }
\date{\today}
\title{Judicious Partitions in Edge-Weighted Graphs with Bounded Maximum Weighted Degree}

\begin{document}

	\date{}
	
	\maketitle
	\begin{abstract}
        In this paper, we investigate bounds for the following
            judicious $k$-partitioning problem: Given an edge-weighted
            graph $G$, find a $k$-partition $(V_1,V_2,\dots ,V_k)$ of
          $V(G)$
            such that the total weight of edges in the heaviest induced
            subgraph, $\max_{i=1}^k w(G[V_i])$, is minimized. In our bounds,
            we also take into account the weight $w(V_1,V_2,\dots,V_k)$ of the
            cut induced by the partition (i.e., the total weight of edges
            with endpoints in different parts) and show the existence of a
            partition satisfying tight bounds for both quantities
            simultaneously.
            We establish such tight bounds for the case $k=2$
                and, to the best of our knowledge, present the first \GG{(even for unweighted graphs)
                completely tight bound for \MA{$k=3$}. We also show that,
                in general, these results cannot be extended to \(k \geq
                4\) without introducing an additional lower-order term, and
                we propose a corresponding conjecture. Moreover, we prove
                that there always exists a \(k\)-partition satisfying  $
\max \left\{ w(G[V_i]) : i \in [k] \right\} \leq \frac{w(G)}{k^2} + \frac{k - 1}{2k^2} \Delta_w(G),
$
where \(\Delta_w(G)\) denotes the maximum weighted degree of \(G\). This bound is tight for every integer $k\geq 2$. 
}
	\end{abstract}

    \section{Introduction}
\GG{For a graph $H$, let $e(H)=|E(H)|.$ A $k$-{\em partition} $(V_1,V_2,\dots ,V_k)$ of a graph $G$ is a partition of $V(G)$ into $k$ disjoint non-empty  parts $V_i$, $i\in [k]$.}
    In a \textit{judicious partitioning problem}, the goal is to find a
    partition of the vertices of a graph which \GG{maximizes/minimizes} several quantities
    simultaneously. The systematic study of the following judicious
    partitioning problem was initiated by Bollob\'{a}s and Scott in
    \cite{BollobasScott93}: Given a graph $G$, find a $k$-partition
    \GG{$(V_1,V_2,\dots ,V_k)$ of $G$} minimizing $\max_{i=1}^k e(G[V_i])$. This
    is related to the well-studied classic Max $k$-Cut problem, which can
    be viewed as asking for a partition minimizing the sum of $e(G[V_i])$
    over $i \in [k]$ instead of the maximum.

    An interesting problem is to determine best-possible upper bounds for
    the objective of these problem in terms of some chosen parameter of the
    input graph (e.g. number of vertices/edges, maximum degree, etc.). This
    problem has been studied extensively for both the above judicious
    partitioning problem
    \cite{BollobasScott93,BollobasScott99,BollobasScott02,BollobasScott04,BollobasScott10,XuYu09,XuYu11,AlonEtAl03}
    and the Max Cut problem
    \cite{BollobasScott02Cut,ErdosGyarfasKohayakawa97,Edwards74}.

    It was shown by Edwards \cite{Edwards74} that any graph $G$ admits a
    bipartition $(V_1,V_2)$  such that the number $ e(V_1,V_2)$ of edges between $V_1$ and $V_2$ has the following lower bound
    \begin{equation}
        \label{eq:edwards}
        e(V_1,V_2) \geq \frac{e(G)}{2} + \sqrt{\frac{e(G)}{8}+\frac{1}{64}} - \frac{1}{8}
    \end{equation}
    and this bound is tight for the complete graphs of odd order.
    Bollob\'{a}s and Scott \cite{BollobasScott99} extended this result to
    show that one can in fact find a bipartition $(V_1,V_2)$ such that both
    Edwards' bound \eqref{eq:edwards} is satisfied and
    $\max\{e(G[V_1]),e(G[V_2])\}$ is small. In particular, they showed the
    following:

    %The judicious partitioning problem for $k=2$ was first studied under
    %the name of the \textit{Bottleneck Bipartition Problem} introduced by
    %Entringer \cite{Porter92}. 

    %Answering a conjecture of
    %Erd\H{o}s, Porter \cite{Porter92} showed that any graph $G$ admits a
    %bipartition $(V_1,V_2)$ such that \[
    %    \max \{e(G[V_1]),e(G[V_2])\} \leq \frac{e(G)}{4} + \sqrt{\frac{e(G)}{8}}.
%\] 

    \begin{thm}\cite{BollobasScott99}\label{thm:BS}
        Any graph $G$ admits a bipartition \GG{$(V_1,V_2)$} such that
        \MA{both}
        \begin{itemize}
        \item[(i)]
            $ e(V_1,V_2) \geq \frac{e(G)}{2} + \sqrt{\frac{e(G)}{8}+\frac{1}{64}} - \frac{1}{8}$ and
            \item[(ii)] $
            \max\{e(G[V_1]),e(G[V_2])\} \leq \frac{e(G)}{4} + \sqrt{\frac{e(G)}{32}+
            \frac{1}{256}}-\frac{1}{16}$.
    \end{itemize}
    \end{thm}

    Note that (ii) is also tight for the complete graphs of odd order.
    Furthermore, Bollob\'{a}s and Scott \cite{BollobasScott99} partially
    extended \YZ{Part (ii) of} Theorem \ref{thm:BS} to $k > 2$ by
    showing that any graphs $G$ admits a $k$-partition  \GG{$(V_1,V_2,\dots , V_k)$} such that 
    \begin{equation}
        \label{eq:BSk}
        \max_{i\in [k]} e(G[V_i]) \leq \frac{e(G)}{k^2} +
        \frac{k-1}{2k^2}\left(\sqrt{2e(G)+\frac{1}{4}}-\frac{1}{2}\right)
    \end{equation}
    which is tight for $K_{kn+1}$ for all positive integers $n$.

    This raises the natural question of whether one can simultaneously
    demand both \eqref{eq:BSk} and a bound analogous to Edwards' bound
    for $k > 2$ \cite[Problem 1]{BollobasScott02}. The
    following analogue to Edward's bound was obtained in
    \cite{BollobasScott02Cut}: Any graph $G$ admits a $k$-partition
    \GG{$( V_1,V_2,\dots,V_k)$} such that \GG{the total number of edges between distinct parts}
    \begin{equation}
        \label{eq:edwardsk}
        e(V_1,V_2,\dots,V_k) \geq \frac{k-1}{k}e(G) +
        \frac{k-1}{2k}\left(\sqrt{2e(G)+\frac{1}{4}} \GG{- \frac{1}{2}}\right) -
        \frac{(k-2)^2}{8k} 
    \end{equation}
    and this is tight for the complete graphs $K_{kn+1}$ for all positive
    integers $n$.\footnote{Note that \cite{BollobasScott02} contains an
    error in \eqref{eq:edwardsk} (see \cite{XuYu11} and
    \cite{WuHou22}).}

    An affirmative answer to whether one can demand both \eqref{eq:BSk} and
    \eqref{eq:edwardsk} simultaneously \GG{subject to an $O(k)$ term in \eqref{eq:edwardsk}} was given by Xu and Yu
    \cite{XuYu09,XuYu11} who showed the following:

    \begin{thm}\cite{XuYu11}\label{thm:XuYu}
        Any graph $G$ admits a $k$-partition \GG{$(V_1,V_2,\dots , V_k)$} such
        that \MA{both}
        \begin{itemize}
            \item[\MA{(i)}] $e(V_1,V_2,\dots,V_k) \geq \frac{k-1}{k}e(G) +
                \frac{k-1}{2k}\left(\sqrt{2e(G)+\frac{1}{4}}-\frac{1}{2}\right)
                -\frac{17k}{8}$ and
            \item[\MA{(ii)}] $\max_{i=1}^k e(G[V_i]) \leq \frac{e(G)}{k^2} +
                \frac{k-1}{2k^2}\left(\sqrt{2e(G)+\frac{1}{4}}-\frac{1}{2}\right)$.
        \end{itemize}
    \end{thm}

    The bound in \GG{Part (i)  of Theorem \ref{thm:XuYu} is }tight up to the $O(k)$ term
    as shown by the complete graphs $K_{nk+1}$ for positive
    integers $n$.
    
\GG{    To the best of our knowledge, there is only one general result for unweighted graphs of bounded maximum degree, which is as follows.

    \begin{thm}\label{thm:maxodddegree}\cite{BollobasScott04} 
For a graph $G$ of maximum odd degree $\Delta(G)$, there is a bipartition
\GG{$(V_1, V_2)$} such that \MA{both}
\begin{itemize}
 \item[(i)]   $e(V_1,V_2)\ge \frac{\Delta(G) + 1}{2\Delta(G)} e(G)$ and
 \item[(ii)] $\max\{e(G[V_1]),e(G[V_2])\}\le \frac{\Delta(G) - 1}{4\Delta(G)} e(G) + \frac{\Delta(G) - 1}{4}.$ 
 \end{itemize}
 \end{thm}

No extension of Theorem \ref{thm:maxodddegree} to $k>2$ has been obtained in the literature. 
 }

    \2

    \subsection{Our contribution}
    \setcounter{equation}{0}
    
    \GG{ This paper continues a line of previous papers \cite{AiGGYZ24,GerkeGYZ24,GutYeo23} which proved many results on the Maximum Weight Cut problem for undirected and directed graphs and the Maximum Weight Bisection problem for undirected 
    graphs. }  \MA{More precisely}, we study the following weighted generalization of the
    above judicious partitioning problem: Given a weighted graph
    $G=(V(G),E(G),w)$ where \GG{$w:E(G)\to \mathbb{R}_{\ge 0}$}, find a $k$-partition \GG{$(V_1,V_2,\dots ,V_k)$ of
    $G$} minimizing $\max_{i=1}^k w(G[V_i])$ where
    $w(G[V_i])$ is the total weight of edges in the induced subgraph
    $G[V_i]$.

    For a weighted graph $G$, the weighted degree of a vertex $v$ is
    defined as $w(v) = \sum_{u\in N(v)} w(v,u)$ and the maximum weighted
    degree is $\Delta_w(G) = \max_{v\in V(G)} w(v)$. Furthermore, the
    average weighted degree $d_w(G)$ is defined by 

\[
        d_w(G) = \AY{ \frac{\sum_{v \in V(G)} w(v)}{|V(G)|} = } \frac{2w(G)}{|V(G)|}.
 \]

    \MA{For disjoint non-empty subsets $X_1,X_2,\dots,X_k \subseteq V$ we denote by $w(X_1,X_2,\dots,X_k)$ the
        total weight of edges with one endpoint in $X_i$ and the other in $X_j$
        for some $i\neq j$. Lastly, for any $X \subseteq V$ and $v \in V$,
        we denote by \[
            w_X(v) = \sum_{x \in X} w(vx)
        \] 
    (where $w(vx) = 0$ if $vx \notin E(G)$).}

    We first consider the case of $k=2$, for which we obtain the following.

    \begin{restatable*}{thm}{maxd}
        \label{thm:maxd}
        Every weighted graph $G=(V(G),E(G),w)$ with order $n\geq 2$ admits
        a bipartition $(X,Y)$ such that \MA{both}
		\begin{itemize}
			\item [(i)] $w(X,Y)\geq \frac{w(G)}{2}+\frac{d_w(G)}{4}$ and
            \item [(ii)] $\max\{w(G[X]),w(G[Y])\}\leq \frac{w(G)}{4}+\frac{\Delta_w(G)}{8}$.
		\end{itemize}
	\end{restatable*}

    As for Theorem \ref{thm:BS} both bounds in Theorem \ref{thm:maxd} are
    tight for unweighted complete graphs of odd order. \GG{In this paper, unweighted graphs correspond to weighted graphs with the same edges and all weights being equal to 1.}
    
\GG{    Note that like Theorem  \ref{thm:maxd}, Theorem
\ref{thm:maxodddegree} is for graphs of bounded maximum degree.  Even though the bounds in Theorems  \ref{thm:maxd}  and \ref{thm:maxodddegree} are formally incomparable, let us provide a partial comparison.  Clearly, Theorem \ref{thm:maxd} deals with a much wider class of graphs. 
    If we compare the non-weighted version of Theorem \ref{thm:maxd} with Theorem \ref{thm:maxodddegree}, then the Part (i) lower bound is stronger in Theorem \ref{thm:maxodddegree} than in Theorem \ref{thm:maxd} unless $\Delta(G)$ is even and $\Delta(G)=|V(G)|-1$, but which of the two Part (ii) upper bounds is stronger
    depends on whether $e(G)$ is larger or smaller than 
    $\frac{\Delta(G)(\Delta(G)-2)}{2}$ if $\Delta(G)$ is odd and  similarly if $\Delta(G)$ is even (in which case we simply replace $\Delta(G)$ by $\Delta(G)+1$) .
    }
    
   %For Part (i), note that  Theorem \ref{thm:maxd} gives the lower bound of $\frac{n(G)+1}{2n(G)}e(G)$, where $n(G)$ is the number of vertices in $G$, and Theorem \ref{thm:maxodddegree} gives $\frac{\Delta(G)+1}{2\Delta(G)}e(G).$ It is easy to see that 
   % the bound of Theorem \ref{thm:maxodddegree} is better. For Part (ii), if $e(G)<\frac{\Delta(G)(\Delta(G)-1)}{1+\Delta(G)/n(G)}$, the bound of Theorem \ref{thm:maxd} is better, otherwise Theorem \ref{thm:maxodddegree} is better or equal. 

\vspace{2mm}

\GG{  Recall that for $k\ge 3$ there are no results extending  Theorem \ref{thm:maxodddegree}.  For $k = 3$, we obtain the following results extending  Theorem \ref{thm:maxd}. }

    \begin{restatable*}{thm}{maxthreetwo}
        \label{thm:max32}
        Every weighted graph $G=(V(G),E(G),w)$ with order $n\geq 3$ admits
        a $3$-partition $(X_1,X_2,X_3)$ such \MA{that both}
		\begin{itemize}
			\item [(i)] $w(X_1,X_2,X_3) \geq \frac{2w(G)}{3} +
                \frac{d_w(G)}{3}$ \MA{and}
			\item [(ii)]  $\max\{w(G[X_1]),w(G[X_2]),w(G[X_3])\}\leq \frac{w(G)}{9}+\frac{\Delta_w(G)}{9}$.
		\end{itemize}
    \end{restatable*}
    
\YZ{To the best of our knowledge, the bound above is the first completely tight bound for judicious \(3\)-partitions. }
\GG{And the unweighted $K_{3\YZ{q}+1}$ (\YZ{$q\ge 1$}) demonstrates sharpness of  both bounds of Theorem \ref{thm:max32} by satisfying them with equality. 
Unfortunately, we \AY{are} unable to get a further extension of Theorem \ref{thm:maxd} \AY{and Theorem \ref{thm:max32}} \YZ{(without introducing a lower-order term)}.  
In Section \ref{sec:disc}, we explain \AY{why this is the case} and \YZ{pose a conjecture}. \YZ{In Section \ref{sec:k-cut}},
by restricting ourselves to only one parameter, we prove the following for arbitrary $k\ge 2$.}

    \begin{restatable*}{thm}{maxk}
        \label{thm:maxk}
        Every weighted graph $G=(V(G),E(G),w)$ with order $n$ admits a
        $k$-partition $(X_1,X_2, \dots,X_k)$ such that $\max\{w(G[X_i]):
        i\in [k]\}\leq \frac{w(G)}{k^2}+\frac{k-1}{2k^2}\Delta_w(G)$. 
	\end{restatable*}

\GG{$K_{qk+1}$ ($q\ge 1$) demonstrates sharpness of the bound of Theorem \ref{thm:maxk}.} \YZ{Note that for unweighted graphs \( G = (V(G), E(G)) \), this bound is tighter than (\ref{eq:BSk}) when \( \Delta^2(G) + \Delta(G) \leq 2e(G) \). For instance, consider the case where \( G \) is an \( r \)-regular graph. In this case, we have
$
\Delta^2(G) + \Delta(G)  = r\cdot (r+1) \leq r \cdot |V(G)|=2e(G).
$}

	\section{Judicious partitioning into two parts}
			\begin{lemma}\label{lem:balancedcut}
		Let $n$ be an integer and \YZ{$t$ be the indicator for whether $n$ is even, i.e., $t=1$ if $n$ is even and 0 otherwise}. Then, the maximum cut of $G$ has weight at least $\left(\frac{1}{2}+\frac{1}{2(n-t)}\right)w(G)$. 
	\end{lemma}
	\begin{proof}
		To see this, choose a bipartition $(X, Y)$ uniformly at random from all possible bisections (i.e., bipartition $(X,Y)$ of $G$ with $||X|-|Y||\leq 1$) of $G$. \YZ{By symmetry, observe that each pair of vertices $x, y \in V(G)$ is in different partite sets with the same probability $p=\frac{\left\lfloor \frac{n}{2} \right\rfloor\cdot \left\lceil \frac{n}{2} \right\rceil}{{n \choose 2}}$. Thus, we are done by the linearity of expectation, as $p =  \frac{n+1}{2n}$ if $n$ is odd and $p= \frac{n}{2(n-1)}$ if $n$ is even. }
		\end{proof}
	%Observe that each pair of vertices lies across the same number of bisections. We denote this number by $N$. Moreover, every bisection contributes exactly $\left\lfloor \frac{n}{2} \right\rfloor\cdot \left\lceil \frac{n}{2} \right\rceil$ vertex pairs that are split across the partition. Let $M$ denote the total number of bisections of $G$. Then, 
	
%	\[
%	{n \choose 2} N = \left\lfloor \frac{n}{2} \right\rfloor \cdot \left\lceil \frac{n}{2} \right\rceil M.
%	\]
	
%	Hence, the probability that a given pair $x,y$ is in different partite sets is
%	\[
%	p = \frac{N}{M} = \frac{\left\lfloor \frac{n}{2} \right\rfloor\cdot \left\lceil \frac{n}{2} \right\rceil}{{n \choose 2}}.
%	\]
%	Thus, if $n$ is odd, then \YZ{$p=\frac{\frac{(n-1)(n+1)}{4}}{{n \choose 2}}=\frac{n+1}{2n}$}. If $n$ is even, then $p=\frac{\frac{n^2}{4}}{{n \choose 2}}=\frac{n}{2(n-1)}$. This completes the proof

    \maxd
	% \begin{thm}\label{thm:maxd}
	% 	Every weighted graph $G=(V(G),E(G),w)$ with order $n$ admits a bipartition $(X,Y)$ such that
	% 	\begin{itemize}
	% 		\item [(i)] $w(X,Y)\geq \frac{w(G)}{2}+\frac{d_w(G)}{4}$;\\
	% 		\item [(ii)] $\max\{w(G[X]),w(G[Y])\}\leq \frac{w(G)}{4}+\frac{\Delta_w(G)}{8}$.
	% 	\end{itemize}
	% \end{thm}
	\begin{proof}
		Let $(X, Y)$ be a maximum weight cut of $G$. Then by Lemma \ref{lem:balancedcut}, we have
		\[
		w(X, Y) \geq \AY{\frac{w(G)}{2}+\frac{w(G)}{2n} =} \frac{w(G)}{2}+\frac{d_w(G)}{4}.
		\]
		For every vertex $v \in X$, we must have $w_X(v) \leq w_Y(v)$, since otherwise moving $v$ from $X$ to $Y$ would increase the cut weight, contradicting the maximality. Similarly, for every $v \in Y$, we must have $w_Y(v) \leq w_X(v)$.
		
		Without loss of generality, assume that $w(G[X]) \geq w(G[Y])$. If
		\[
		w(G[X]) \leq \frac{w(G)}{4} + \frac{\Delta_w(G)}{8},
		\]
		we are done. Thus, assume instead that
		\[
		w(G[X]) > \frac{w(G)}{4} + \frac{\Delta_w(G)}{8}.
		\]
		
		Let $(X', Y')$ (where $X'\subseteq X$ and $Y\subseteq Y'$) be a bipartition obtained from $(X, Y)$ by recursively moving vertices from $X$ to $Y$ \AY{until} the following conditions hold:
		\begin{itemize}
			\item [(a)] $w(G[X'])>\frac{w(G)}{4}+\frac{\Delta_w(G)}{8}$;\\
			\item [(b)] There exists a vertex $v\in X'$, $w(G[X'\setminus\{v\}])\leq \frac{w(G)}{4}+\frac{\Delta_w(G)}{8}$.
		\end{itemize}
		
		By (a), and the fact that for all $u \in X'$, $w_{Y'}(u) \geq w_Y(u) \geq w_X(u) \geq w_{X'}(u)$, we have that
		\begin{eqnarray*}
			w(X'\setminus \{v\}, Y'\cup\{v\})&=&\sum_{u\in X'\setminus\{v\}} w_{Y'\cup \{v\}}(u)\\
			&=& (\sum_{u\in X'\setminus \{v\}} w_{Y'}(u)) + w_{X'}(v)\\
			&\geq & (\sum_{u\in X'\setminus \{v\}} w_{X'}(u)) + w_{X'}(v)\\
			&=& 2w(G[X'])\\
			&>& \frac{w(G)}{2}+\frac{\Delta_w(G)}{4}. \\
		\end{eqnarray*}
		
		It follows from (a) and the above inequality that
		\begin{eqnarray*}
			w(Y'\cup \{v\})&=&w(G)-w(X'\setminus \{v\}, Y'\cup\{v\})-w(G[X'])+w_{X'}(v)\\
			&<&\frac{w(G)}{4}-\frac{3\Delta_w(G)}{8}+\frac{w(v)}{2}\\
			&\leq &\frac{w(G)}{4}-\frac{3\Delta_w(G)}{8}+\frac{\Delta_w(G)}{2}\\
			&=& \frac{w(G)}{4}+\frac{\Delta_w(G)}{8}, 
		\end{eqnarray*}
		where the first inequality holds as $w_{X'}(v)\leq w_{Y'}(v)$. This together with (b) implies that 
		\[\max\{w(G[X'\setminus \{v\}]), w(G[Y'\cup \{v\}])\}\leq \frac{w(G)}{4}+\frac{\Delta_w(G)}{8}. \]
		This completes the proof since $w(X'\setminus \{v\}, Y'\cup\{v\})>\frac{w(G)}{2}+\frac{\Delta_w(G)}{4}\geq \frac{w(G)}{2}+\frac{d_w(G)}{4}$. 
	\end{proof}
	
	Note that both bounds (i) and (ii) are tight when $ G $ is an unweighted complete graph with an odd number of vertices. Moreover, the term $ \frac{d_w(G)}{4} $ in the first bound cannot, in general, be replaced by $ \frac{\Delta_w(G)}{4} $. In fact, as the following proposition shows, we cannot guarantee a cut of weight at least $ \frac{w(G)}{2} + c \cdot \Delta_w(G) $, even for arbitrarily small constants $ c > 0 $.
	
	\begin{prop}
		For any positive real number $ c > 0 $, there exists a weighted graph $ G=(V(G),E(G),w) $ such that the weight of every cut in $ G $ is less than $ \frac{w(G)}{2} + c \cdot \Delta_w(G) $.
	\end{prop} 
	
	\begin{proof}
		Fix any positive real number $ c \leq \frac{1}{4} $, and let $ n $ be an integer greater than $ 1/4c^2$. Construct a graph $ G $ by starting with the complete graph $ K_n $, assigning weight 1 to all its edges, and then adding a new vertex $ v $ connected to every vertex in $ K_n $ with edges of weight $ 2cn $. Note that $2cn > 1/2c\geq 2 $ as $n>1/4c^2$.
		
		In this construction:
		\[
		\Delta_w(G) = 2cn^2 \quad \text{and} \quad w(G) = 2cn^2 + \binom{n}{2}.
		\]
		
		Now consider a maximum weight cut $(X, Y)$. Suppose $ v \in X $, and let $ a $ be the number of vertices from $ K_n $ that are in $Y$, so $ 1 \leq a \leq n $. The weight of cut is now
		\[
		f(a) = a(2cn + n - a),
		\]
		which is a quadratic function in $ a $ and therefore bounded above by its maximum value $(c+\frac{1}{2})^2n^2$. Thus,
		\begin{eqnarray*}
			\frac{w(G)}{2}+c\cdot\Delta_w(G)-f(a)&\geq&\frac{w(G)}{2}+c\cdot\Delta_w(G)-(c+\frac{1}{2})^2n^2\\
			&=&cn^2+\frac{n(n-1)}{4}+2c^2n^2-c^2n^2-cn^2-\frac{n^2}{4}\\
			&=& c^2n^2-\frac{n}{4}\\
			&>& \frac{n}{4}-\frac{n}{4}=0,
		\end{eqnarray*}
		where the last inequality holds as $n>1/4c^2$, completing the proof. 
	\end{proof}
	
Conversely, the following proposition can be proved by considering a graph obtained by adding a sufficient number of isolated vertices to a complete graph. 
 
 \begin{prop}
 		For any positive real number $ c > 0$, there exists a weighted graph $ G=(V(G),E(G),w) $ such that for all cuts $(X,Y)$ of $G$, $\max\{w(G[X]),w(G[Y])\}> \frac{w(G)}{4}+c\cdot d_w(G)$. 
 \end{prop}

\section{Judicious 3-partition}

\YZ{The following lemma can be shown by a similar argument as in Lemma \ref{lem:balancedcut}.}
	
	\begin{lemma}\label{lem:max-3-cut}
	Every weighted graph $G=(V(G),E(G),w)$ with order $n$ admits a $3$-partition $(X_1,X_2,X_3)$ \AY{of $V(G)$} with  $w(X_1,X_2,X_3) \geq \frac{2w(G)}{3} + \frac{d_w(G)}{3}$. 
   \end{lemma}
	
   \maxthreetwo
	%\begin{thm}\label{thm:max32}
        % Every weighted graph $G=(V(G),E(G),w)$ with order $n\geq 3$ admits a $3$-partition $(X_1,X_2,X_3)$ \AY{of $V(G)$} such that both the following holds.
		% \begin{itemize}
		% 	\item [(i)] $w(X_1,X_2,X_3) \geq \frac{2w(G)}{3} + \frac{d_w(G)}{3}$. 
		% 	\item [(ii)]  $\max\{w(G[X_1]),w(G[X_2]),w(G[X_3])\}\leq \frac{w(G)}{9}+\frac{\Delta_w(G)}{9}$.
		% \end{itemize}
	%\end{thm}
	
	\begin{proof}
		 Let $(X_1,X_2,X_3)$ be a maximum weight $3$-partition of $G$. That is $w(X_1,X_2)+w(X_1,X_3)+w(X_2,X_3)$ is maximum possible. Without loss of generality assume that $w(G[X_3]) \geq \max\{w(G[X_1]),w(G[X_2])\}$. Thus, we may assume that $w(G[X_3]) > \frac{w(G)}{9}+\frac{\Delta_w(G)}{9}$ as otherwise, by Lemma \ref{lem:max-3-cut}, $(X_1,X_2,X_3)$ is a desired 3-cut.
		
		Now, recursively remove vertices from $X_3$ until we have a subset $X_3^*\subseteq X_3$ such that 
		
			\begin{itemize}
			\item [(a)] $w(G[X_3^*])> \frac{1}{9}w(G)+\frac{1}{9}\Delta_w(G)$
			\item [(b)] For \AY{every} vertex $y\in X_3^*$,  $w(G[X_3^*\setminus\{y\}])\leq  \frac{1}{9}w(G)+\frac{1}{9}\Delta_w(G)$. 
		\end{itemize}
		
        Let $x$ be a vertex with $w_{X_3^*}(x)=\Delta_w(G[X_3^*])$,
        $X_3'=X_3^*\setminus\{x\}$ and $n_3=|X_3'|$. Note that by the
        definition of a $3$-cut $|X_i|\geq 1$ for all $i\in [3]$. Thus,
		
		\begin{equation}\label{eq:lbforn-n3}
            n-n_3\geq |X_1|+|X_2|+|\{x\}| \geq 3
		\end{equation}
		
		By (b), we may define $r\geq 0$ as the real number such that
		
		\begin{equation}\label{eq:w(G[X_3])}
			w(G[X_3'])=\frac{1}{9}w(G)+\frac{1}{9}\Delta_w(G)-r
		\end{equation} 
		
		Thus, as $w_{X_3^*}(x)=\Delta_w(G[X_3^*])$, 	$w_{X_3^*}(x)\geq \frac{2w(G[X_3^*])}{|X_3^*|}= \frac{2(w(G[X_3'])+w_{X_3'}(x))}{n_3+1}$, which, together with (\ref{eq:w(G[X_3])}), implies the following \YZ{if $n_3\geq 2$}.
		
		\begin{equation}\label{eq:w_{X_3'}(x)>=}
			w_{X_3'}(x)\geq \frac{2}{n_3-1}w(G[X_3'])=\frac{2}{9(n_3-1)}w(G)+\frac{2}{9(n_3-1)}\Delta_w(G)-\frac{2}{n_3-1}r
		\end{equation}
		
		Since $(X_1,X_2,X_3)$ is a maximum $3$-cut of $G$, for every $y\in
        X_3$ and $i\in \{1,2\}$, moving $y$ to $X_i$ should not increase
        the weight of the $3$-cut, and therefore, as $X_3'\subseteq X_3^*\subseteq X_3$ we have that
		
		\begin{equation}\label{eq: w(X_3)(y) is small}
			w_{X_3'}(y)\leq  w_{X_3^*}(y)\leq  w_{X_3}(y)\leq w_{X_i}(y)
		\end{equation}

        The following claims now holds. 		
		\begin{claim}\label{cl:r}
			$0\leq r < w_{X_3'}(x)\leq \Delta_w(G)/3$. 
		\end{claim}
		
		\begin{proof}
			On the one hand, by (a) and (\ref{eq:w(G[X_3])}), $w_{X_3'}(x)=w_{X_3^*}(x)=w(G[X_3^*])-w(G[X_3'])>r$. On the other hand, by (\ref{eq: w(X_3)(y) is small}), $3\cdot w_{X_3'}(x)\leq \sum_{i=1}^3 w_{X_i}(x)=w(x)$ and therefore $w_{X_3'}(x)\leq w(x)/3\leq \Delta_w(G)/3$. This completes the proof of Claim \ref{cl:r}. 
		\end{proof}
		
		By Claim \ref{cl:r}, we can define $0\leq c<3$ as the constant such that
		
		\begin{equation}\label{eq:r}
		 r=\frac{c}{9}\Delta_w(G)
		\end{equation}
		
		\2
		
		We define $\theta$ as the constant such that
		
		\begin{equation}\label{eq:w(X_3',V-X_3')}
				w(X_3',V(G) \setminus X_3') = \frac{4}{9}w(G) + \frac{4}{9}\Delta_w(G) - r + \theta.
			\end{equation}
		
		Then, the following two claims gives lower bounds for $\theta$. 
		
		\begin{claim}\label{cl:lbtheta}
		$\theta \geq 3(w_{X_3'}(x)-r)\geq  0$.
		\end{claim}
		\begin{proof}
			Note that we only need to show $\theta \geq 3(w_{X_3'}(x)-r)$ as $w_{X_3'}(x)-r\geq  0$ follows from Claim \ref{cl:r}. By (\ref{eq: w(X_3)(y) is small}), 
			
			\[
                w(X_3', X_1\cup X_2) \MA{\ =} \sum_{y\in X_3'} \left(\sum_{i=1}^{2} w_{X_i}(y)\right)
			\geq 2\sum_{y\in X_3'} w_{X_3^*}(y)= 2\cdot (2w(G[X_3'])+w_{X_3'}(x)).
			\]
			
			Thus, by (\ref{eq:w(G[X_3])})
			\[
			\begin{array}{rcl} \2
								w(X_3',V(G)\setminus X_3')&\geq& w(X_3', X_1\cup X_2)+w_{X_3'}(x)\\ \2
				&\geq& 4w(G[X_3'])+3w_{X_3'}(x)\\ \2
				& = &  4 \left( \frac{w(G)}{9}+\frac{\Delta_w(G)}{9} - r
                \right) + \MA{3w_{X_3'}(x)}\\ \2
				& = & \frac{4w(G)}{9} + \frac{4\Delta_w(G)}{9} - r + 3(w_{X_3'}(x)-r), \\
			\end{array}
			\]
			which completes the proof of Claim \ref{cl:lbtheta}. 
		\end{proof}
		
			Recall that $n_3=|X_3'|$. 
		
		\begin{claim}\label{cl:lbtheta2}
			\YZ{If $n_3\geq 2$, then} $\frac{\theta}{2}\left(1-\frac{1}{n-n_3}\right)\geq \frac{n-n_3-1}{3(n-n_3)(n_3-1)}w(G)+\frac{(2-c(n_3+1))(n-n_3-1)}{6(n-n_3)(n_3-1)}\Delta_w(G)$.
		\end{claim}
		\begin{proof}
			By Claim \ref{cl:lbtheta} and (\ref{eq:w_{X_3'}(x)>=}) and (\ref{eq:r}), we have that
			\[
			      \begin{array}{rcl}\2
				\theta\geq  3(w_{X_3'}(x)-r)&\geq& \frac{2}{3(n_3-1)}w(G)+\frac{2}{3(n_3-1)}\Delta_w(G)-\frac{3(n_3+1)}{n_3-1}r\nonumber \\\2
				&\geq& \frac{2}{3(n_3-1)}w(G)+\frac{2-c(n_3+1)}{3(n_3-1)}\Delta_w(G)\nonumber,
			\end{array}
			\]
			and therefore
			\begin{equation*}
				\frac{\theta}{2}\left(1-\frac{1}{n-n_3}\right)\geq \frac{n-n_3-1}{3(n-n_3)(n_3-1)}w(G)+\frac{(2-c(n_3+1))(n-n_3-1)}{6(n-n_3)(n_3-1)}\Delta_w(G),
			\end{equation*}
			which completes the proof of Claim \ref{cl:lbtheta2}. 
		\end{proof}

		\2

       Let $G' = G - X_3'$. Then, by (\ref{eq:w(G[X_3])}) and (\ref{eq:w(X_3',V-X_3')}), the weight of $G'$ is as follows:

		\begin{equation}\label{eq:G'}
			w(G')  =  w(G) - w(G[X_3']) - w(X_3',V(G) \setminus X_3') 
			 =  \frac{4w(G)}{9} - \frac{5\Delta_w(G)}{9} + 2r - \theta
	\end{equation}
		
		\2
		
	We now show the following claim which gives \AY{a} lower bound for the maximum weighted degree of $G$. 
		
		\begin{claim}\label{cl:lbdelta}
			$\Delta_w(G)\geq \max\{\frac{2w(G)+3\theta}{3n_3-2+c}, \frac{4w(G)-3\theta}{3n-3n_3+2-c}\}$. 
		\end{claim}
		
		\begin{proof}
			\AY{Note that} the maximum weighted degree of \( G \), \( \Delta_w(G) \), is at least the average weighted degrees of the vertices in \( X_3'\). Therefore, by (\ref{eq:w(G[X_3])}), (\ref{eq:r}) and (\ref{eq:w(X_3',V-X_3')}), we obtain the following: 
			\[
			\begin{array}{rcl}\2
				n_3\cdot\Delta_w(G)&\geq& \sum_{v\in X_3'} w(v)\\\2
				&=& w(X_3',V(G)\setminus X_3')+2w(G[X_3'])\\\2
				&=& \left(\frac{4w(G)}{9}+\frac{4\Delta_w(G)}{9} - r+\theta\right)+2\left(\frac{w(G)}{9}+\frac{\Delta_w(G)}{9} - r\right)\\\2
				&=& \frac{2}{3}w(G)+\frac{2}{3}\Delta_w(G)-3r+\theta\\\2
				&=& \frac{2w(G)+3\theta}{3}+\frac{2-c}{3}\Delta_w(G),
			\end{array}
			\]
			which implies that $\Delta_w(G)\geq \frac{2w(G)+3\theta}{3n_3-2+c}$.

			Similarly, as \( \Delta_w(G) \) is at least the average weighted degrees of the vertices in \( V(G)\setminus X_3' \). Therefore, we obtain the following: 
			\[
			\begin{array}{rcl}\2
				(n-n_3)\cdot\Delta_w(G)&\geq& \sum_{v\in V(G)\setminus X_3'} w(v)\\\2
				&=& 2w(G)-w(X_3',V(G)\setminus X_3')-2w(G[X_3'])\\\2
				&=& 2w(G)-\left(\frac{4w(G)}{9}+\frac{4\Delta_w(G)}{9} - r+\theta\right)-2\left(\frac{w(G)}{9}+\frac{\Delta_w(G)}{9} - r\right)\\\2
				&=& \frac{4}{3}w(G)-\frac{2}{3}\Delta_w(G)+3r-\theta\\\2
				&=& \frac{4w(G)-3\theta}{3}-\frac{2-c}{3}\Delta_w(G),
			\end{array}
			\]
			which implies that $\Delta_w(G)\geq \frac{4w(G)-3\theta}{3n-3n_3+2-c}$. This completes the proof of Claim \ref{cl:lbdelta}. 
		\end{proof}
		
		Let $(X_1',X_2')$ be a maximum cut of $G'$. Then, the following claim holds.
		
		\begin{claim}\label{cl:w(x1x2x3)}
Let \YZ{$t'$ be the indicator for whether $n-n_3$ is even, i.e., $t'=1$ if $n-n_3$ is even and 0 otherwise}. Then,
\[
\begin{array}{rcl}\2
	&~&w(X_1',X_2',X_3')-\frac{2}{3}w(G)-\frac{d_w(G)}{3}\\\2
	&\geq& \frac{3(n-n_3-t')-5+2c}{18(n-n_3-t')}\Delta_w(G)-\frac{4n-6n_3-6t'}{9n(n-n_3-t')}w(G)+\frac{\theta}{2} \left(1- \frac{1}{n-n_3-t'} \right)\\\2
	&\geq& \frac{3(n-n_3)-5+2c}{18(n-n_3)}\Delta_w(G)-\frac{4n-6n_3}{9n(n-n_3)}w(G)+\frac{\theta}{2} \left(1- \frac{1}{n-n_3} \right).
\end{array}
  \]
		\end{claim}
		\begin{proof}
			As $(X_1',X_2')$ is a maximum cut of $G'$, by Lemma \ref{lem:balancedcut}, \YZ{(\ref{eq:r}), (\ref{eq:w(X_3',V-X_3')}) and (\ref{eq:G'})} we have that
			\[
					\begin{array}{rcl}\2
				&~&w(X_1',X_2',X_3')- \frac{2}{3}w(G)-\frac{d_w(G)}{3}\\\2
				& = & w(X_1',X_2') + w(X_3',V(G) \setminus X_3')-\frac{2}{3}w(G)-\frac{2w(G)}{3n}\\ \2
				& \geq & \left( \frac{w(G')}{2} + \YZ{\frac{w(G')}{2(n-n_3-t')} } \right) +  \left( \frac{4w(G)}{9} + \frac{4\Delta_w(G)}{9} - r + \theta \right) -\frac{2}{3}w(G)-\frac{2w(G)}{3n}\\ \2
				& = & \frac{w(G')}{2} + \frac{w(G')}{2(n-n_3-t')} - \frac{2w(G)}{9} -\frac{2w(G)}{3n} + \frac{4\Delta_w(G)}{9} - r + \theta \nonumber\\ \2
				&\YZ{=} & \left( \frac{4w(G)}{9} - \frac{5\Delta_w(G)}{9} + 2r - \theta \right) \cdot \left( \frac{1}{2} + \frac{1}{2(n-n_3-t')} \right) \\ \2
				&~&- \frac{2w(G)}{9} -\frac{2w(G)}{3n} + \frac{4\Delta_w(G)}{9} - r + \theta \nonumber\\  \2
					&=& \frac{3(n-n_3-t')-5}{18(n-n_3-t')}\Delta_w(G)-\frac{4n-6n_3\YZ{-6t'}}{9n(n-n_3-t')}w(G)+\frac{r}{n-n_3-t'}+\frac{\theta}{2} \left(1- \frac{1}{n-n_3-t'} \right)\\ \2
					&=& \frac{3(n-n_3-t')-5+2c}{18(n-n_3-t')}\Delta_w(G)-\frac{4n-6n_3-6t'}{9n(n-n_3-t')}w(G)+\frac{\theta}{2} \left(1- \frac{1}{n-n_3-t'} \right)\\
						&\geq& \frac{3(n-n_3)-5+2c}{18(n-n_3)}\Delta_w(G)-\frac{4n-6n_3}{9n(n-n_3)}w(G)+\frac{\theta}{2} \left(1- \frac{1}{n-n_3} \right),
			\end{array}
			\]
			where the last inequality can be seen from the \YZ{third equality} as $t'\geq 0$. This completes the proof of Claim \ref{cl:w(x1x2x3)}. 
		\end{proof}
		
		Define $i$ such that $n = 3 n_3 + i$. The following claim now holds.
		
		\begin{claim}\label{cl:w(x1x2x3)-2}
			\YZ{If $n_3\geq 2$, then} $w(X_1',X_2',X_3')-\frac{2}{3}w(G)-\YZ{\frac{d_w(G)}{3}}$ is at least the following
				\[
	 \frac{(1-c)\cdot(3(n-n_3-1)(n_3+1)-2(n_3-1))}{18(n-n_3)(n_3-1)}\Delta_w(G)+\frac{12n_3^2+(11i-3)n_3+3i^2+i}{9n(n-n_3)(n_3-1)} w(G).
			\]
		\end{claim}
	    \begin{proof}
	    	By Claims \ref{cl:lbtheta2} and \ref{cl:w(x1x2x3)}, we have that
	    	\[
	    			\begin{array}{rcl}\2
	    		&~&w(X_1',X_2',X_3')-\frac{2}{3}w(G)-\YZ{\frac{d_w(G)}{3}} \nonumber\\ \2
	    		&\geq& \frac{3(n-n_3)-5+2c}{18(n-n_3)} \Delta_w(G)-\frac{4n-6n_3}{9n(n-n_3)}w(G)+\frac{\theta}{2}(1-\frac{1}{n-n_3})\\ \2
	    		&=& \frac{(3(n-n_3-1)-2+2c)(n_3-1)+6(n-n_3-1)-3c(n_3+1)(n-n_3-1)}{18(n-n_3)(n_3-1)}\Delta_w(G)\nonumber\\ \2
	    		&~& + \frac{3n(n-n_3-1)-(4n-6n_3)(n_3-1)}{9n(n-n_3)(n_3-1)}w(G)\nonumber\\ \2
	    		&=& \frac{(1-c)\cdot(3(n-n_3-1)(n_3+1)-2(n_3-1))}{18(n-n_3)(n_3-1)}\Delta_w(G)+\frac{n(3n-7n_3+1)-6n_3+6n_3^2}{9n(n-n_3)(n_3-1)} w(G)\nonumber\\ \2
	    		&=& \frac{(1-c)\cdot(3(n-n_3-1)(n_3+1)-2(n_3-1))}{18(n-n_3)(n_3-1)}\Delta_w(G)+\frac{12n_3^2+(11i-3)n_3+3i^2+i}{9n(n-n_3)(n_3-1)} w(G)\nonumber,
	    	\end{array}
	    	\]
	    	which completes the proof of Claim \ref{cl:w(x1x2x3)-2}.
	    \end{proof}
		
		\begin{claim}\label{cl:lb3cut}
			 $w(X_1',X_2',X_3') \geq \frac{2w(G)}{3} + \frac{d_w(G)}{3}$.
		\end{claim}
		\begin{proof}
        		We now consider the follwong cases.
        		
        			{\bf Case 1, $ i \geq 1$:} 
               By Claims \ref{cl:lbtheta}, \ref{cl:lbdelta} and \ref{cl:w(x1x2x3)}, we have the following holds, where the second inequality holds as the derivative of the function $f(x)=\frac{3(n-n_3)-5+2x}{18(n-n_3)(3n_3-2+x)}$ is given by $f'(x)=\frac{1-3i}{18(n-n_3)(3n_3-2+x)^2}$, which is negative for all real values $x\geq 0$\YZ{, and therefore} as $0\leq c<3$, $f(c)\geq f(3)$. 
               \[
               \begin{array}{rcl}\2
               		&~&w(X_1',X_2',X_3')-\frac{2}{3}w(G)-\frac{d_w(G)}{3}\\ \2
               	&\geq& \frac{3(n-n_3)-5+2c}{18(n-n_3)}\cdot \frac{2w(G)+3\theta}{3n_3-2+c}-\frac{4n-6n_3}{9n(n-n_3)}w(G)+\frac{\theta}{2} \left(1- \frac{1}{n-n_3} \right)\\ \2
               	&\geq& \frac{3(n-n_3)+1}{9(n-n_3)}\cdot \frac{w(G)}{3n_3+1}-\frac{4n-6n_3}{9n(n-n_3)}w(G)\\ \2
               		&\geq& \frac{(3(2n_3+i)+1)(3n_3+i)-(6n_3+4i)(3n_3+1)}{9n(n-n_3)(3n_3+1)}w(G)\\ \2
               		&=& \frac{(6n_3+4i+1-i)(3n_3+i)-(6n_3+4i)(3n_3+1)}{9n(n-n_3)(3n_3+1)}w(G)\\ \2
               		&=& \frac{(6n_3+4i)(i-1)+(1-i)(3n_3+i)}{9n(n-n_3)(3n_3+1)}w(G)\\ \2 
               		&=& \frac{3(n_3+i)(i-1)}{9n(n-n_3)(3n_3+1)}w(G)\\ \2
               		&\geq& 0, 
               \end{array}
               \]
which completes the proof of Case 1. 
        		
        	{\bf Case 2, $ i\leq -5$:} By Claims \ref{cl:lbtheta}, \ref{cl:w(x1x2x3)}, and the fact that $\Delta_w(G)\geq d_w(G)=\frac{2w(G)}{n}$, we have that
        	\[
        		\begin{array}{rcl}\2
        			&~&w(X_1',X_2',X_3')-\frac{2}{3}w(G)-\YZ{\frac{d_w(G)}{3}} \nonumber\\ \2
        			&\geq& \frac{3(n-n_3)-5+2c}{18(n-n_3)} \Delta_w(G)-\frac{4n-6n_3}{9n(n-n_3)}w(G)+\frac{\theta}{2}\left(1-\frac{1}{n-n_3}\right)\\\2
        			&\geq& \frac{3(n-n_3)-5}{18(n-n_3)}\cdot\frac{2w(G)}{n} -\frac{4n-6n_3}{9n(n-n_3)}w(G)\nonumber\\\2
        			&=& \frac{3(n-n_3)-5-4n+6n_3}{9n(n-n_3)}w(G) \nonumber \\\2
        			&=& \frac{-i-5}{9n(n-n_3)}w(G) \nonumber \\\2
        			&\geq & 0, \nonumber
        		\end{array}
        		\]
        		which completes the proof of Case 2. 
        		
        			{\bf Case 3, $ i= 0$ or $-2$:} In this case, $n-n_3$ is even. Thus, by (\ref{eq:lbforn-n3}) and Claims \ref{cl:lbtheta}, \ref{cl:lbdelta} and \ref{cl:w(x1x2x3)}, we have that
        			\[
        		\begin{array}{rcl} \2
        			&~&w(X_1',X_2',X_3')-\frac{2}{3}w(G)-\YZ{\frac{d_w(G)}{3}} \\ \2
        			&\geq& \frac{3(n-n_3-1)-5+2c}{18(n-n_3-1)} \Delta_w(G)-\frac{4n-6n_3-6}{9n(n-n_3-1)}w(G)
        			+ \frac{\theta}{2} \left( 1 - \frac{1}{n-n_3-1} \right)\\ \2
        			&\geq& \frac{3(n-n_3-1)-5+2c}{9(n-n_3-1)} \cdot  \frac{w(G)}{3n_3-2+c} -\frac{4n-6n_3-6}{9n(n-n_3-1)}w(G).
        		\end{array}
        		\]

        		Now consider the function  
        		$
        		f(x) = \frac{3(n - n_3 - 1) - 5 + 2x}{9(n - n_3 - 1)(3n_3 - 2 + x)}.
        		$
        		Its derivative is given by  
        		$
        		f'(x) = \frac{4 - 3i}{9(n - n_3 - 1)(3n_3 - 2 + x)^2},
        		$
        		which is positive for all real values of \( x\geq 0 \). Therefore, since \( c \geq 0 \), it follows that \( f(c) \geq f(0) \), and thus we have:
        		\[
        		\begin{array}{rcl} \2
        			&~&w(X_1',X_2',X_3')-\frac{2}{3}w(G)-\YZ{\frac{d_w(G)}{3}} \\ \2
        			&\geq& \frac{3(n-n_3-1)-5}{9(n-n_3-1)} \cdot  \frac{w(G)}{3n_3-2} -\frac{4n-6n_3-6}{9n(n-n_3-1)}w(G)\\ \2
        		%	&= & \frac{3(n-n_3-1)-5}{9(n-n_3-1)} \cdot  \frac{w(G)}{3n_3-2} -\frac{4n-6n_3-6}{9n(n-n_3-1)}w(G)\\ \2
        			&= & \frac{3n^2-3nn_3-8n-12nn_3+18n_3^2+18n_3+8n-12n_3-12}{9n(n-n_3-1)(3n_3-2)}w(G) \\ \2
        			&= & \frac{n^2-5nn_3+6n_3^2+2n_3-4}{3n(n-n_3-1)(3n_3-2)}w(G) \\ \2
        			&= & \frac{(n-2n_3)^2+2n_3^2-(3n_3+i)n_3+2n_3-4}{3n(n-n_3-1)(3n_3-2)}w(G) \\ \2
        			&= & \frac{(n_3+i)^2+2n_3^2-(3n_3+i)n_3+2n_3-4}{3n(n-n_3-1)(3n_3-2)}w(G) \\ \2
        			&= & \frac{(i +2)n_3+ i^2-4}{3n(n-n_3-1)(3n_3-2)}w(G), 
        		\end{array}
        		\]
        		which is non-negative as $i\in \{0,-2\}$. This completes the proof of Case 3. 
        		
        		{\bf Case 4, $ i= -4$:} In this case $n-n_3$ is even. Thus, by \YZ{(\ref{eq:lbforn-n3}) and} Claims \ref{cl:lbtheta}, \ref{cl:lbdelta} and \ref{cl:w(x1x2x3)}, we have that
        		\[
        		\begin{array}{rcl} \2
        			&~&w(X_1',X_2',X_3')-\frac{2}{3}w(G)-\YZ{\frac{d_w(G)}{3}} \\ \2
        			&\geq& \frac{3(n-n_3-1)-5+2c}{18(n-n_3-1)} \Delta_w(G)-\frac{4n-6n_3-6}{9n(n-n_3-1)}w(G)
        			+ \frac{\theta}{2} \left( 1 - \frac{1}{n-n_3-1} \right)\\ \2
        			&\geq& \frac{3(n-n_3-1)-5+2c}{18(n-n_3-1)} \cdot   \frac{4w(G)-3\theta}{3n-3n_3+2-c} -\frac{4n-6n_3-6}{9n(n-n_3-1)}w(G)+ \frac{\theta}{2}\left( 1 - \frac{1}{n-n_3-1} \right)\\ \2
        		%	&\geq& \frac{3(n-n_3-1)-5+2c}{18(n-n_3-1)} \cdot   \frac{4w(G)-3\theta}{3n-3n_3+2-c} -\frac{4n-6n_3-6}{9n(n-n_3-1)}w(G) 
        		%	+ \frac{\theta}{2} \left(1 - \frac{1}{n-n_3-1} \right)\\ \2
        			&=&\frac{6(n-n_3-1)-10+4c}{9(n-n_3-1)(3n-3n_3+2-c)}  w(G) -\frac{4n-6n_3-6}{9n(n-n_3-1)}w(G) \\ \2
        			&~& + \frac{\theta}{2} \left( 1 - \frac{1}{n-n_3-1}-\frac{3(n-n_3-1)-5+2c}{3(n-n_3-1)(3n-3n_3+2-c)} \right)\\
        			&\geq & \frac{6(n-n_3-1)-10}{9(n-n_3-1)(3n-3n_3+2)}  w(G) -\frac{4n-6n_3-6}{9n(n-n_3-1)}w(G) \\ \2
        			&~& + \frac{\theta}{2} \left( 1 - \frac{1}{n-n_3-1}-\frac{3n-3n_3-2}{3(n-n_3-1)(3n-3n_3-1)} \right)\\ \2
        			&\geq & \frac{6(n-n_3-1)-10}{9(n-n_3-1)(3n-3n_3+2)}  w(G) -\frac{4n-6n_3-6}{9n(n-n_3-1)}w(G) + \frac{\theta}{2} \left( 1 - \frac{4}{3(n-n_3-1)} \right)\\ \2
        			&\geq & \frac{6n^2-6nn_3-16n-(4n-6n_3-6)(3n-3n_3+2)}{9n(n-n_3-1)(3n-3n_3+2)} w(G), 
        		\end{array}
        		\]
        		where the last inequality follows from \YZ{Claim \ref{cl:lbtheta} and} (\ref{eq:lbforn-n3}), and the third last inequality holds as $0\leq c< 3$. Now, we only need to show $6n^2-6nn_3-16n-(4n-6n_3-6)(3n-3n_3+2)\geq 0$. Indeed,
        		\[
        		\begin{array}{rcl} \2
        			&~&	6n^2-6nn_3-16n-(4n-6n_3-6)(3n-3n_3+2)\\ \2
        			&=&  6n^2-6nn_3-16n-12n^2+18nn_3+18n+12nn_3-18n_3^2-18n_3-8n+12n_3+12\\ \2
        			&= & -6(n-2n_3)^2+6n_3^2-6n_3-6n+12\\ \2
        			&= & -6(n_3-4)^2+6n_3^2-6n_3-6(3n_3-4)+12\\ 
        			&= & 24n_3-60. 
        		\end{array}
        		\]
        		Note that by (\ref{eq:lbforn-n3}), $n-n_3=2n_3-4\geq 3$ and therefore $n_3\geq 4$. Thus, $24n_3-60>0$, which completes the proof of Case 4.

        		\2
        		
        		{\bf Case 5, $ i= -1$ or $-3$:} \YZ{In this case, $n_3\geq 2$ as otherwise $n=3n_3+i\leq 3-1<3$, a contradiction. We first narrow our focus by examining the lower bound provided in Claim~\ref{cl:w(x1x2x3)-2}.}
        		
        		As $n_3\geq 2$, according to Claim~\ref{cl:w(x1x2x3)-2}, we have:
        		\[
        		\begin{array}{rcl} \2
        			w(X_1', X_2', X_3') - \frac{2}{3}w(G) - \YZ{\frac{d_w(G)}{3}}  
        			&\geq \frac{(1 - c)\cdot \big(3(n - n_3 - 1)(n_3 + 1) - 2(n_3 - 1)\big)}{18(n - n_3)(n_3 - 1)} \Delta_w(G) \\ \2
        			&\quad + \frac{12n_3^2 + (11i - 3)n_3 + 3i^2 + i}{9n(n - n_3)(n_3 - 1)} w(G).
        		\end{array}
        		\]
        		
        		Multiplying both sides of this inequality by \(18n(n - n_3)(n_3 - 1)\), we define:
        		\[
        		\alpha = (1 - c)\cdot n \cdot \big(3(n - n_3 - 1)(n_3 + 1) - 2(n_3 - 1)\big) \Delta_w(G) + 2\big(12n_3^2 + (11i - 3)n_3 + 3i^2 + i\big) w(G).
        		\]
        		
        		Note that if \(\alpha \geq 0\), then it follows that
        		\[
        		w(X_1', X_2', X_3') \geq \frac{2}{3}w(G) + \YZ{\frac{d_w(G)}{3}} ,
        		\]
        		in which case the proof is complete. Therefore, we proceed under the assumption that \(\alpha < 0\). 
        		
        	  \YZ{It is not hard to check that as $i\in \{-1,-3\}$ and $n_3\geq 2$, $12n_3^2 + (11i - 3)n_3 + 3i^2 + i\geq 0$. Thus, since \(\alpha < 0\) and the term \(3(n - n_3 - 1)(n_3 + 1) - 2(n_3 - 1)\) is clearly non-negative, $1-c<0$ and therefore $c>1$. }

        		As \YZ{$\alpha<0$} and $c>1$, the following result now holds:
        		\begin{eqnarray*}
        			\Delta_w(G)>\frac{2(12n_3^2+(11i-3)n_3+3i^2+i)}{(c-1)n(3(n-n_3-1)(n_3+1)-2(n_3-1))}w(G).
        		\end{eqnarray*}
        		Furthermore, when $i=-1$, it translate to the following.
        		\begin{eqnarray}\label{eq:i=-1}
        			\Delta_w(G)&>&\frac{4(6n_3^2-7n_3+1)}{(c-1)n(6(n_3-1)(n_3+1)-2(n_3-1))}w(G)\nonumber\\
        			&=& \frac{2(n_3-1)(6n_3-1)}{(c-1)\YZ{n}(n_3-1)(3n_3+2)} w(G)\nonumber\\
        			&=& \frac{2(6n_3-1)}{(c-1)\YZ{n}(3n_3+2)} w(G)
        		\end{eqnarray}
        		
        			Now we consider the cases $i=-1$ and $i=-3$ seperately.
        		
        		\noindent{\bf Subcase 5.1, $i=-3$:} By Claims \ref{cl:lbtheta}, \ref{cl:lbdelta} and \ref{cl:w(x1x2x3)}, we have that
        		\[
        		\begin{array}{rcl}\2
        			&~&w(X_1',X_2',X_3')-\frac{2}{3}w(G)-\frac{2}{3n}w(G) \nonumber\\ \2
        			&\geq& \frac{3(n-n_3)-5+2c}{18(n-n_3)} \Delta_w(G)-\frac{4n-6n_3}{9n(n-n_3)}w(G)+\frac{\theta}{2}(1-\frac{1}{n-n_3})\\ \2
        			&\geq& \frac{3(n-n_3)-5+2c}{18(n-n_3)} \cdot \frac{4w(G)-3\theta}{3n-3n_3+2-c}-\frac{4n-6n_3}{9n(n-n_3)}w(G)+\frac{\theta}{2}(1-\frac{1}{n-n_3})\\ \2
        			&\geq& \frac{n-n_3-1}{6(n-n_3)(3n-3n_3+1)} \cdot (4w(G)-3\theta)-\frac{4n-6n_3}{9n(n-n_3)}w(G)+\frac{\theta}{2}(1-\frac{1}{n-n_3})\\ \2
        			&=& \frac{2n_3-4}{3(2n_3-3)(3n_3-4)}\cdot w(G) -\frac{2n_3-4}{3(2n_3-3)(3n_3-3)}w(G)\\ \2
        			&~&+\frac{\theta}{2}(1-\frac{1}{n-n_3}-\frac{n-n_3-1}{(n-n_3)(3n-3n_3+1)})\\ \2
        			&\geq & \frac{\theta}{2}(1-\frac{1}{n-n_3}-\frac{1}{3(n-n_3)})\\ \2
        			&\geq & 0,
        		\end{array}
        		\]
        		where the third inequality holds as $c>1$, and the last inequality holds as $n-n_3\geq 3$ which has been shown in (\ref{eq:lbforn-n3}). This completes the proof of Subcase 5.1. 
        		
        		\noindent{\bf Subcase 5.2, $i=-1$:} We discuss the following two cases based on the value of $c$. As $\frac{9n_3}{3n_3+1}>2$, the following two subcases exhaust all possibilities.
        		
        		\noindent{\bf Subcase 5.2.1, $0<c-1\leq  \frac{6n_3-1}{3n_3+1}$:} By (\ref{eq:i=-1}), we have the following holds, where the second inequality holds as $f(x)=\frac{(3(n-n_3-1)+2x)(6n_3-1)}{9n(n-n_3)(3n_3+2)x}$ is clearly decreasing when $x>0$. 
        		\[
        		\begin{array}{rcl}\2
        			\frac{3(n-n_3)-5+2c}{18(n-n_3)} \Delta_w(G)&>& 	\frac{3(n-n_3)-5+2c}{18(n-n_3)}\cdot \frac{2(6n_3-1)}{(c-1)n(3n_3+2)} w(G)\\ \2
        			&=&\frac{(3(n-n_3-1)+2(c-1))(6n_3-1)}{9n(n-n_3)(3n_3+2)(c-1)} w(G)\\ \2
        			&\geq  &\frac{(3(n-n_3-1)+2\cdot \frac{6n_3-1}{3n_3+1})(6n_3-1)}{9n(n-n_3)(3n_3+2)\cdot \frac{6n_3-1}{3n_3+1}}w(G) \\ \2
        			&=& \frac{6(n_3-1)(3n_3+1)+2\cdot (6n_3-1)}{9n(n-n_3)(3n_3+2)} w(G)\\ \2
        			&=& \frac{18n_3^2-8}{9n(n-n_3)(3n_3+2)} w(G)\\ \2
        			&=& \frac{(3n_3+2)(6n_3-4)}{9n(n-n_3)(3n_3+2)} w(G)\\ \2
        			&=& \frac{4n-6n_3}{9n(n-n_3)} w(G), 
        		\end{array}
        		\]
        		which, together with Claims \ref{cl:lbtheta} and \ref{cl:w(x1x2x3)}, implies $w(X_1',X_2',X_3')\geq \frac{2}{3}w(G)+\YZ{\frac{d_w(G)}{3}} $. This completes the proof of Subcase 5.2.1.
        		
        		\noindent{\bf Subcase 5.2.2, $c\geq 2$:} By Claims \ref{cl:lbtheta}, \ref{cl:lbdelta} and \ref{cl:w(x1x2x3)}, we have the following holds.
        		\[
        		\begin{array}{rcl} \2
        			&~&w(X_1',X_2',X_3')-\frac{2}{3}w(G)-\frac{2}{3n}w(G) \nonumber\\ \2
        			&\geq& \frac{3(n-n_3)-5+2c}{18(n-n_3)} \Delta_w(G)-\frac{4n-6n_3}{9n(n-n_3)}w(G)+\frac{\theta}{2}\left(1-\frac{1}{n-n_3}\right)\\ \2
        			&\geq& \frac{3(n-n_3)-5+2c}{18(n-n_3)} \cdot \frac{4w(G)-3\theta}{3n-3n_3+2-c}-\frac{4n-6n_3}{9n(n-n_3)}w(G)+\frac{\theta}{2}\left(1-\frac{1}{n-n_3}\right)\\ \2
        			&\geq& \frac{2(3(n-n_3)-1)}{27(n-n_3)^2}w(G) -\frac{4n-6n_3}{9n(n-n_3)}w(G)+\frac{\theta}{2}\left(1-\frac{1}{n-n_3}-\frac{3(n-n_3)-1}{9(n-n_3)^2}\right)\\ \2
        			&\geq& \frac{2(3(n-n_3)-1)}{27(n-n_3)^2}w(G) -\frac{4n-6n_3}{9n(n-n_3)}w(G)+\frac{\theta}{2}\left(1-\frac{1}{n-n_3}-\frac{1}{3(n-n_3)}\right)\\ \2
        			&\geq& \frac{2}{9(n-n_3)}\left(\frac{3(n-n_3)-1}{3(n-n_3)} -\frac{2n-3n_3}{n}\right)w(G)\\ \2
        			&=& \frac{2}{9(n-n_3)}\left(\frac{6n_3-4}{6n_3-3} -\frac{6n_3-4}{6n_3-2}\right)w(G)\\ \2
        			&\geq & 0,
        		\end{array}
        		\]
        		where the second last inequality follows from (\ref{eq:lbforn-n3}). This completes the proof of Case 5 and therefore the proof of Claim \ref{cl:lb3cut}.
		\end{proof}

		\2
		
        Now we are ready to prove this result. If $\max\{w(G[X_1']), w(G[X_2'])\} \leq \frac{w(G)}{9}+\frac{\Delta_w(G)}{9}$ then
		we would be done by Claim \ref{cl:lb3cut}. So without loss of generality assume that $w(G[X_2']) \geq w(G[X_1'])$ and 
		$w(G[X_2']) > \frac{w(G)}{9}+\frac{\Delta_w(G)}{9}$.
		
		While $w(G[X_2']) > \frac{w(G)}{9}+\frac{\Delta_w(G)}{9}$ take any vertex from $X_2'$ and put it into $X_1'$.
		Let $u$ be the last vertex we transfered from $X_2'$ to $X_1'$ and let $X_1''$ and $X_2''$ denote the new $X_1'$-set and
		$X_2'$-set respectively. Let $r'$ be defined such that $w(G[X_2'']) = \frac{w(G)}{9}+\frac{\Delta_w(G)}{9} - r'$ and
		note that $r' \geq 0$ and $w_{X_2''}(u) \geq r'$. As $(X_1',X_2')$ is a maximum weighted cut in $G'$ we note that 
		for any vertex in $v \in X_2'$ we have $w_{X_1'}(v) \geq w_{X_2'}(v)$. This implies that the following holds,
		
		\[
		w(X_2'',X_1') \geq \sum_{v \in X_2''} w_{X_2''\cup\{u\}}(v) 
		\geq 2w(G[X_2'']) + w_{X_2''}(u) 
		\geq 2 \left( \frac{w(G)}{9}+\frac{\Delta_w(G)}{9} - r' \right) + r'.
		\]
		
		Furthermore as $w_{X_2''}(u) \geq r'$ we get that $w(X_2'',X_1'')\geq w(X_2'',X_1')+w_{X_2''}(u) \geq \frac{2w(G)}{9}+\frac{2\Delta_w(G)}{9}$.
		This implies the following by Claim \ref{cl:r}, Claim \ref{cl:lbtheta} and (\ref{eq:w(X_3',V-X_3')}).  
		
		\[
		w(X_2'',X_1'',X_3') \geq \frac{6w(G)}{9} + \frac{6\Delta_w(G)}{9} - r + \theta \geq \frac{2w(G)}{3} + \frac{2\Delta_w(G)}{3} 
		- \frac{\Delta_w(G)}{3} = \frac{2w(G)}{3} + \frac{\Delta_w(G)}{3}.
		\]
		
		So if $w(G[X_1'']) \leq \frac{w(G)}{9}+\frac{\Delta_w(G)}{9}$, then we would be done, as $(X_2'',X_1'',X_3')$ would be the 
		desired partition (as $\Delta_w(G) \geq d_w(G)$). Note that the following holds.
		
		\[
		\begin{array}{rcl} \2
			w(G[X_1'']) & = & w(G) - w(X_2'',X_1'',X_3') - w(G[X_2'']) - w(G[X_3'])  \\ \2
			& \leq & w(G) - \left( \frac{6w(G)}{9} + \frac{6\Delta_w(G)}{9} - r + \theta \right) 
			- \left( \frac{w(G)}{9}+\frac{\Delta_w(G)}{9} - r' \right)
			- \left( \frac{w(G)}{9}+\frac{\Delta_w(G)}{9} - r \right) \\ \2
			& = & \frac{w(G)}{9} - \frac{8\Delta_w(G)}{9} + 2r + r' - \theta \\
			& \leq & \frac{w(G)}{9} - \frac{8\Delta_w(G)}{9} + 2r + r'. \\
		\end{array}
		\]
		
		If $r' \leq \frac{\Delta_w(G)-r}{2}$, then the above implies the following, as $r \leq \frac{\Delta_w(G)}{3}$ by Claim \ref{cl:r}.
		
		\[
		w(G[X_1'']) \leq \frac{w(G)}{9} - \frac{8\Delta_w(G)}{9} + 2r + \frac{\Delta_w(G)-r}{2} \leq \frac{w(G)}{9} + \frac{\Delta_w(G)}{9}.
		\]
		
		We may therefore assume that $r' > \frac{\Delta_w(G)-r}{2}$.
		However, when we transfered $u$ from $X_2'$ to $X_1'$ we noted that $w_{X_1'}(u) \geq w_{X_2'}(u) \geq r'$,
		which implies that $w_{V(G')}(u) \geq 2r' > \Delta_w(G)-r$. So $w_{X_3'}(u) < r$.
		We will now transfer $u$ from $X_1''$ to $X_3'$ and call the new sets $X_1^*$ and $X_3^*$. Let $X_2^*=X_2''$.
		Now the following holds. 
		
		\begin{itemize}
			\item $w(G[X_3^*]) = \frac{w(G)}{9}+\frac{\Delta_w(G)}{9} - r + w_{X_3'}(u) < \frac{w(G)}{9}+\frac{\Delta_w(G)}{9}$.
			\item $w(G[X_2^*]) = w(G[X_2'']) = \frac{w(G)}{9}+\frac{\Delta_w(G)}{9} - r' \leq \frac{w(G)}{9}+\frac{\Delta_w(G)}{9}$. 
			\item $w(G[X_1^*]) = w(G[X_1'']) - w_{X_1''}(u) \leq w(G[X_1'']) - r' \leq \frac{w(G)}{9} - \frac{8\Delta_w(G)}{9} + 2r \leq \frac{w(G)}{9}$.
		\end{itemize}
		
		Furthermore the following holds as, by Claim \ref{cl:r}, $r \leq \Delta_w(G)/3$.
		
		\[
		\begin{array}{rcl} \2
			w(X_1^*,X_2^*,X_3^*) & = & w(X_2'',X_1'',X_3') - w_{X_3'}(u) + w_{X_1''}(u) \\ \2
			& \geq & w(X_2'',X_1'',X_3') - r + r' \\ \2
			& > & \left(\frac{2w(G)}{3} + \frac{\Delta_w(G)}{3} \right) - r + \frac{\Delta_w(G)-r}{2} \\ \2
			& = & \frac{2w(G)}{3} + \frac{\Delta_w(G)}{3} + \frac{\Delta_w(G)}{2} - \frac{3r}{2} \\
			& \geq & \frac{2w(G)}{3} + \frac{\Delta_w(G)}{3}. \\
		\end{array}
		\]
		Thus, $(X_1^*,X_2^*,X_3^*)$ is now the desired 3-partition, completing the proof.
	\end{proof}
	
	Theorem~\ref{thm:max32} is best possible due to the following example.  Let $G = K_{3\YZ{q}+1}$ where \YZ{$q \geq 1$} is an integer and
	let all weights be 1 (i.e. consider the unweighted case).
	That is $G$ is a complete graph of order $3q+1$. It is not difficult to see that the optimal $3$-partition, $(X_1,X_2,X_3)$, of $G$ is 
	when the minimum size of a set $X_i$ and the maximum size of a set $X_j$ differ by at most one.  That is, we may assume 
	that $|X_1| = |X_2|=q$ and $|X_3|=q+1$.  In this case $w(G[X_3]) = {q+1 \choose 2}$ and 
	$w(G)= {3q+1 \choose 2}$ and $\Delta_w(G)=3q$. the following now holds.
	
	\[
	\frac{w(G) + \Delta_w(G)}{9} =  \frac{3q(3q+1)/2 + 3q}{9} = \frac{q^2 + q}{2} = w(G[X_3]).
	\]
	
	And furthermore the following also holds as $d_w(G)=3q$ and $w(G)=\frac{3q(3q+1)}{2}$,
	
	\[
	w(X_1,X_2,X_3) = q^2 + q(q+1) + q(q+1) =  \frac{3q(3q+1)}{3} + \frac{3q}{3} = \frac{2w(G)}{3} + \frac{d_w(G)}{3}.
	\]

	  \section{$k$-partitions}\label{sec:k-cut}
	
      \maxk
	% \begin{thm}\label{thm:maxk}
	% 	Every weighted graph $G=(V(G),E(G),w)$ with order $n$ admits a $k$-partition $(X_1,X_2, \dots,X_k)$ such that
	% 	$\max\{w(G[X_i]): i\in [k]\}\leq \frac{w(G)}{k^2}+\frac{k-1}{2k^2}\Delta_w(G)$. 
	% \end{thm}
	\begin{proof}
		We prove the result by induction on $k\ge 2$. When $k=2$, the result follows from Theorem \ref{thm:maxd}. Assume that $k\ge 3$ and the result holds for all integers in the interval $[2,k-1]$. Now, let $(X_1,X_2, \dots, X_k)$ be a maximum weight $k$-partition of $G$ where  $\sum_{1\leq i<j\leq k}w(X_i,X_j)$ is maximized. Without loss of generality assume that $w(G[X_1])=\max\{w(G[X_i]): i\in [k]\}$. If $w(G[X_1]) \leq\frac{w(G)}{k^2}+\frac{k-1}{2k^2}\Delta_w(G)$ then we are done, so assume this is not the case.
		
		Repeatedly remove vertices from $X_1$ until $w(G[X_1]) \leq \frac{w(G)}{k^2}+\frac{k-1}{2k^2}\Delta_w(G)$.
		Call the resulting set $X_1'$ and let $x$ be the last vertex removed from $X_1$ before obtaining $X_1'$.
		Define $r$ to be the value such that the following holds
		
		\[
		w(G[X_1']) = \frac{w(G)}{k^2}+\frac{k-1}{2k^2}\Delta_w(G) - r.
		\]
		
		Note that $r \geq 0$ and $r \leq w_{X_1'}(x)$ as when creating $X_1'$ we noted that $x$ was the last vertex removed from $X_1$. 
		%	Therefore $w(G[X_1']) \geq \frac{w(G)}{9}+\frac{\Delta_w(G)}{9} - w_{X_3'}(x)$. 
		
		As $(X_1,X_2, \dots, X_k)$ is a maximum weight $k$-partition of $G$ we note that for any vertex $u \in X_1$ we have that for any $t\neq 1$, $w_{X_1}(u) \leq w_{X_t}(u)$ (as otherwise a larger partition could be obtained by moving $u$ to $X_t$). Therefore, $w_{X_1'}(x) \leq w_{X_1}(x) \leq w(x)/k$, which implies that $r \leq w(x)/k \leq \Delta_w(G)/k$. 
		
		As for any vertex $u \in X_1$ and $t\neq 1$, $w_{X_1}(u) \leq w_{X_t}(u)$ we note that the following holds.
		\begin{eqnarray*}
			w(X_1', V(G)\setminus X_1) &\geq& (k-1) \sum_{u \in X_1'} w_{X_1} (u) \\
			&\geq& (k-1) (2w(G[X_1']) + w(\{x\}, X_1')) \\
			&\geq& 2(k-1) w(G[X_1']) + (k-1)r.
		\end{eqnarray*}
		
		This implies that following:
		
		\[
		\begin{array}{rcl} \2
			w(X_1',V(G) \setminus X_1') & \geq & w(X_1', V(G)\setminus X_1) + w(X_1',\{x\}) \\ \2
			& \geq &  2(k-1)  w(G[X_1']) + kr
		\end{array}
		\]
		
		Let $G' = G - X_1'$ and note that the following holds.
		
		\[
		\begin{array}{rcl} \2
			w(G') & = & w(G) - w(G[X_1']) - w(X_1',V(G) \setminus X_1') \\ \2
			&\leq & w(G)-(2k-1)w(G[X_1'])-kr\\
			& = & w(G) -(2k-1) \left( \frac{w(G)}{k^2}+\frac{k-1}{2k^2}\Delta_w(G) - r \right) - kr \\
			& = & \left(\frac{k-1}{k}\right)^2w(G)-\frac{(k-1)(2k-1)}{2k^2}\Delta_w(G)+(k-1)r  \\
		\end{array}
		\]
		
		We now consider the following two cases.
		
		\2
		
		{\bf Case 1: $\Delta_w(G') \leq \Delta_w(G) - r$.} 
		By the induction hypothesis, $G'$ has a $k-1$-partition $(X_2',X_3',\dots, X_k')$ with
		$\max\{w(G'[X_i']): 2\leq i\leq k\}\leq \frac{w(G')}{(k-1)^2}+\frac{k-2}{2(k-1)^2}\Delta_w(G')$. This implies the following, as $r \leq \Delta_w(G)/k$.
		
		\[
		\begin{array}{rcl} \2
			\max\{w(G'[X_i']): 2\leq i\leq k\} & \leq & \frac{w(G')}{(k-1)^2}+\frac{k-2}{2(k-1)^2}\Delta_w(G')\\ &\leq& \frac{w(G)}{k^2}-\frac{2k-1}{2k^2(k-1)}\Delta_w(G)+\frac{r}{k-1}+\frac{k-2}{2(k-1)^2}(\Delta_w(G)-r)\\ 
			&\leq&\frac{w(G)}{k^2}+\frac{k^3-4k^2+3k-1}{2k^2(k-1)^2}\Delta_w(G)+\frac{k}{2(k-1)^2}r\\
			&\leq&\frac{w(G)}{k^2}+\frac{k^3-4k^2+3k-1}{2k^2(k-1)^2}\Delta_w(G)+\frac{1}{2(k-1)^2}\Delta_w(G)\\
			&=& \frac{w(G)}{k^2}+\frac{k-1}{2k^2}\Delta_w(G)
		\end{array}
		\]
		
		The  $k$-partition $(X_1',X_2', \dots,X_k')$ of $G$ now satisfies the desired condition, so Case 1 is completed.

		{\bf Case 2: $\Delta_w(G') > \Delta_w(G) - r$.}  In this case let $y$ be a vertex in $G'$, such that $w_{G'}(y) > \Delta_w(G) - r$.
		Let $X_1'' = X_1' \cup \{y\}$ and let $G'' = G - X_1''$. As $w_{G'}(y) > \Delta_w(G) - r$, we note that $w_{X_1''}(y) < r$, which 
		implies that the following holds.
		
		\[
		w(G[X_1'']) =\frac{w(G)}{k^2}+\frac{k-1}{2k^2}\Delta_w(G) - r + w_{X_1''}(y) < \frac{w(G)}{k^2}+\frac{k-1}{2k^2}\Delta_w(G).
		\]
		
		Now the following holds.
		
		\[  
		\begin{array}{rcl} \2
			w(G'') & =  & w(G') - w_{G'}(y) \\ \2
			& < & \left(\frac{k-1}{k}\right)^2w(G)-\frac{(k-1)(2k-1)}{2k^2}\Delta_w(G)+(k-1)r- (\Delta_w(G) - r) \\
			& = & \left(\frac{k-1}{k}\right)^2w(G) -\frac{4k^2-3k+1}{2k^2}\Delta_w(G) + kr \\
		\end{array}
		\] 
		
		By the induction hypothesis, $G''$ has a $k-1$-partition $(X_2'',X_3'',\dots, X_k'')$ with
		$\max\{w(G'[X_i'']): 2\leq i\leq k\}\leq \frac{w(G'')}{(k-1)^2}+\frac{k-2}{2(k-1)^2}\Delta_w(G'')$. This implies the following, as $r \leq \Delta_w(G)/k$.
		
		\[  
		\begin{array}{rcl} \2
			\max\{w(G'[X_i'']): 2\leq i\leq k\}&\leq& \frac{w(G'')}{(k-1)^2}+\frac{k-2}{2(k-1)^2}\Delta_w(G'')\\
			&\leq & \frac{w(G)}{k^2} -\frac{4k^2-3k+1}{2k^2(k-1)^2}\Delta_w(G) + \frac{1}{(k-1)^2}\Delta_w(G)+ \frac{k-2}{2(k-1)^2}\Delta_w(G)\\
			&\leq & \frac{w(G)}{k^2} + \frac{k^3-4k^2+3k-1}{2k^2(k-1)^2}\Delta_w(G)\\
			&\leq & \frac{w(G)}{k^2} + \frac{k^3-3k^2+3k-1}{2k^2(k-1)^2}\Delta_w(G)\\
			&=& \frac{w(G)}{k^2}+\frac{k-1}{2k^2}\Delta_w(G)
		\end{array}
		\] 
		
		We have now shown that $(X_1'',X_2'', \dots, X_k'')$ is a $k$-partition of $G$ and the following holds, as desired.
		
		\[
		\max\{w(G'[X_i'']): 2\leq i\leq k\}\leq \frac{w(G)}{k^2}+\frac{k-1}{2k^2}\Delta_w(G).
		\]
	\end{proof}
	
	Theorem~\ref{thm:maxk} is best possible due to the following example.  Let $G = K_{qk+1}$ where $q \geq 1$ is an integer and
	let all weights be 1 (i.e. consider the unweighted case).
	%That is $G$ is a complete graph of order $qk+1$. 
	It is not difficult to see that the optimal $k$-partition, $(X_1,X_2, \dots,X_k)$, of $G$ is 
	when the minimum size of a set $X_i$ and the maximum size of a set $X_j$ differ by at most one.  That is, we may assume 
	that $|X_1| = k+1$ and $|X_i|=k$ for all $2\leq i\leq k$.  In this case $w(G[X_1]) = {q+1 \choose 2}$ and 
	$w(G)= {qk+1 \choose 2}$ and $\Delta_w(G)=qk$. The following now holds.
	
	\[
	\frac{w(G)}{k^2}+\frac{k-1}{2k^2}\Delta_w(G) =  \frac{qk(qk+1)}{2k^2} + \frac{k-1}{2k^2}kq = \frac{q^2 + q}{2} = w(G[X_1]).
	\]

\section{Discussions}\label{sec:disc}

A natural question is whether Theorem \ref{thm:maxk} can be extended to a judicious partition for $k \geq 4$, as we did for Theorem \ref{thm:max32}. A natural approach would be to combine the best possible lower bound for a $k$-cut with the upper bound from Theorem \ref{thm:maxk}.

The following lower bound can be established for a $k$-cut using an argument similar to that in Lemma \ref{lem:balancedcut}.

\begin{thm}\label{thm:kkkk}
	Every weighted graph $G=(V(G), E(G), w)$ of order $n$ has a $k$-cut with weight at least
	\[
	\frac{k-1}{k}w(G)+\frac{k-1}{2k}\left(1-\frac{(k-2)^2}{4(n-1)(k-1)}\right)d_w(G)
	\]
	if $k$ is even, and
	\[
	\frac{k-1}{k}w(G)+\frac{k-1}{2k}\left(1-\frac{k-3}{4(n-1)}\right)d_w(G)
	\]
	if $k$ is odd.
\end{thm}

Note that $k = 2$ and $k = 3$ are the only values for which the lower-order term vanishes.

This bound is tight for the complete graph, when the number of vertices satisfies \( n \equiv \lfloor\frac{k}{2}\rfloor \) or \( \lceil\frac{k}{2}\rceil \mod k \). To see this, let \( n = kq + t \) and consider the complete graph \( K_n \). The optimal $k$-cut corresponds to a $k$-partition \( (X_1, X_2, \dots, X_k) \), where \( t \) parts have \( q+1 \) vertices and the remaining \( k - t \) parts have \( q \) vertices. The weight of this $k$-cut is then:

\[
\begin{array}{rcl}
	w(X_1,X_2,\dots,X_k) &=& \frac{t(q+1)(n - q - 1) + (k - t)q(n - q)}{2} \\
	&=& \frac{n(n - q - 1) + (k - t)q}{2} \\
	&=& \frac{n((k - 1)q + t - 1) + (k - t)q}{2}.
\end{array}
\]

Thus,

\[
\begin{array}{rcl}
	\frac{w(X_1,X_2,\dots, X_k)}{w(G)}&=&\frac{k-1}{k}\frac{\frac{k}{k-1}\cdot (n((k-1)q+t-1)+(k-t)q)}{n(n-1)}\\
	&=&\frac{k-1}{k}\frac{ n(kq+t-1)+\frac{t-1}{k-1}n+\frac{k-t}{k-1}(n-t)}{n(n-1)}\\
	&=&\frac{k-1}{k}\left(1+\frac{ \frac{t-1}{k-1}n+\frac{k-t}{k-1}(n-t)}{n(n-1)}\right)\\
	&=&\frac{k-1}{k}\left(1+\frac{1}{n}\cdot\frac{ (t-1)n+(k-t)(n-t)}{(n-1)(k-1)}\right)\\
	&=&\frac{k-1}{k}\left(1+\frac{1}{n}\cdot\frac{ (k-1)n-t(\YZ{k}-t)}{(n-1)(k-1)}\right)\\
	&\geq &\frac{k-1}{k}\left(1+\frac{1}{n}\cdot\frac{ (k-1)n-\lfloor\frac{k}{2}\rfloor\cdot\lceil\frac{k}{2}\rceil}{(n-1)(k-1)}\right),
\end{array}
\]
which implies that $w(X_1,X_2,\dots ,X_k)\geq \frac{k-1}{k}w(G)+\frac{k-1}{2k}\left(1-\frac{(k-2)^2}{4(n-1)(k-1)}\right)d_w(G)$ when $k$ is even and $w(X_1,X_2,\dots ,X_k)\geq \frac{k-1}{k}w(G)+\frac{k-1}{2k}\left(1-\frac{k-3}{4(n-1)}\right)d_w(G)$ when $k$ is odd. Moreover, the equality holds if and only if $n\equiv \lfloor\frac{k}{2}\rfloor$ or $\lceil\frac{k}{2}\rceil$ (mod $k$). \YZ{Similarly, the bound is tight for all weighted complete graphs of the same order, where each edge has the same weight.}

It can also be shown—using Brooks' Theorem and a random $k$-coloring on $\chi(G)$ independent sets—\YZ{that these are the only tight examples for Theorem \ref{thm:kkkk}.}

Based on this, we propose the following conjecture:

\begin{conj}
	Every weighted graph \( G = (V(G), E(G), w) \) of order \( n \) admits a $k$-partition \( (X_1, X_2, \dots, X_k) \) satisfying both:
	\begin{itemize}
		\item[(i)] \( w(X_1,X_2, \dots, X_k) \geq \frac{k-1}{k}w(G) + \frac{k-1}{2k}(1 - h)d_w(G) \), where
		\[
		h =
		\begin{cases}
			\frac{(k - 2)^2}{4(n - 1)(k - 1)}, & \text{if } k \text{ is even}, \\
			\frac{k - 3}{4(n - 1)}, & \text{if } k \text{ is odd};
		\end{cases}
		\]
		\item[(ii)] \( \max\{w(G[X_i]) : i \in [k]\} \leq \frac{w(G)}{k^2} + \frac{k - 1}{2k^2} \Delta_w(G) \).
	\end{itemize}
\end{conj}

However, as (ii) is not tight for extremal examples for condition (i) when \( k \geq 4 \), unlike the case \( k \leq 3 \), there is no graph for which both (i) and (ii) are simultaneously tight when \( k \geq 4 \).

\end{document}